\documentclass[11pt]{amsart}
\usepackage{microtype}
\usepackage{amssymb}
\usepackage{amsthm}
\usepackage{amscd}
\usepackage{calc}
\usepackage{hyperref}
\usepackage{cite}
\newtheorem{theorem}{Theorem}[section]
\newtheorem{lemma}[theorem]{Lemma}

\theoremstyle{definition}
\newtheorem{definition}[theorem]{Definition}

\makeatletter

\def\dotminussym#1#2{%
  \setbox0=\hbox{$\m@th#1-$}%
  \kern.5\wd0%
  \hbox to 0pt{\hss\hbox{$\m@th#1-$}\hss}%
  \raise.6\ht0\hbox to 0pt{\hss$\m@th#1.$\hss}%
  \kern.5\wd0}

\mathchardef\mhyphen="2D

\allowdisplaybreaks[2]

\newcommand{\fn}[1]{\widehat{\mathbf{#1}}}

\begin{document}

\title[A Lower Bound on the James Space]{An Inverse Ackermannian Lower Bound on the Local Unconditionality Constant of the James Space}
\author{Henry Towsner}
\date{\today}
%\thanks{Partially supported by NSF grant DMS-1157580.}
%\address {Department of Mathematics, University of Pennsylvania, 209 South 33rd Street, Philadelphia, PA 19104-6395, USA}
%\email{htowsner@math.upenn.edu}
%\urladdr{\url{www.math.upenn.edu/~htowsner}}

\begin{abstract}
The proof that the James space is not locally unconditional appears to be non-constructive, since it makes use of an ultraproduct construction.  Using proof mining, we extract a constructive proof and obtain a lower bound on the growth of the local unconditionality constants.
\end{abstract}

\maketitle

% We should be able to drop Hahn-Banach by using the functional $\lim_p e^*_p$, which is a linear function
% I have no idea how to handle modeling $\lambda^{\leq}_k(d_\infty)$, though

% Here's where we are:
% \begin{itemize}
% \item On the first part---local unconditionality---the key piece seems to be showing that if we a collection of operators $T_{p,q}$ on $L_\infty$ with the property that $\lim_p T_{p,q}$ exists for each $q$ and $\lim_q T_{p,q}$ exists for each $p$ (in practice, these are the functions $\int_B e^*_qe_p$) then $\lim_p \lim_q T_{p,q}=\lim_q\lim_p T_{p,q}$.  This can be proven using Radon-Nikodym, but can we prove it without?
% \item On the second part, we have the main argument, but we want a non-ultraproduct proof that derivatives are additive (ironically, we don't need existence, except insofar as it's needed to prove additivity).  We have an idea using the proof that finite dentability index implies finite tree property to prove a slightly more general ``finite weighted bush'' property.  We would then like to show that the non-existence of finite weighted bushes implies differentiability of Lipschitz functions.  But while it's easy to get somewhere differentiability, it's not so clear how to get a.e. differentiability (and, looking at the proof in GNFA, it may really be hard---the way they get from somewhere differentiability to a.e. involves Hahn-Banach!)
% \end{itemize}

\section{Introduction}

The failure of local unconditionality of the James space has been given \cite{16329} as an example of a theorem whose only known proof requires an ultraproduct argument (or, equivalently, nonstandard analysis).  Recent developments in proof mining provide new insight into the relationship between standard and nonstandard proofs \cite{MR2964881}.  In particular, proof mining techniques can be used to ``extract'' a standard proof from a nonstandard one---that is, they provide a syntactic transformation which converts nonstandard proofs to standard ones.

In this paper with give a concrete example of these developments: a proof of the failure of local unconditionality of the James space which is explict and constructive, and which provides the first (but very slow) lower bound on the local unconditionality constant of subspaces of the James space.  This illustrates the following features of the modern understanding of the role of ultraproducts in proofs:
\begin{itemize}
\item Proofs of standard theorems which use ultraproducts can be systematically converted to explicit, constructive proofs which do not make use of ultraproducts.
\item The main reason ultaproduct methods simplify proofs is that they allow the use of non-constructive theorems which have high \emph{quantifier complexity}.  In this example, the crucial step is a statement about the exchange of the order of a double limit.
\item The functional interpretation provides a way to interpret statements about the ultraproduct as more complicated statements about the original space.  %In this case we are able to interpret the statement about double limits as a statement about ``metastable convergence of limits'' as described below.
\end{itemize}

The usual proof of the local unconditionality of the James space has a measure theoretic flavor (in particular, it makes use of some of the theory of $L_1$-space).  To produce a constructive proof, we need a finitary version of these measure theoretic notions.  The basic methods to do this were developed in the context of ergodic theory, most notably Tao's quantitative ergodic theory \cite{tao06}.  The most important idea is replacing convergence of limits with \emph{metastable convergence} as introduced in \cite{tao:MR2408398,avigad:MR2550151}.  Formally, a sequence converges when
\[\forall\epsilon\exists n_\epsilon\forall M\geq n |a_{n_\epsilon}-a_M|<\epsilon.\]
The metastable version of this statement is
\[\forall\epsilon, F\exists n_{\epsilon,F}\forall m\in[n_{\epsilon,F},F(n_{\epsilon,F})]|a_{n_{\epsilon,F}}-a_m|<\epsilon.\]
That is, given an accuracy $\epsilon$ and a function $F$, we can find an interval $[n,F(n)]$ on which the sequence is close to stable.  The key point is that, in general, $n_\epsilon$ may not be computable from $\epsilon$, but $n_{\epsilon,F}$ usually is computable from $\epsilon$ and $F$.  Importantly, if we use the convergence of a sequence as an intermediate step in a proof, the actual bounds on our final theorem depend only on $n_{\epsilon,F}$ for a suitably chosen $F$.

The general relationship between statements and their constructive versions is given by the proof theoretic functional interpretation \cite{avigad:MR1640329,kohlenbach:MR2445721}.  In particular, the functional interpretation tells us how to convert more complicated statements into the corresponding constructive versions.  The precise technique used to produce the results in this paper is an informal version of Kohlenbach's monotone functional interpretation \cite{MR1195271}.  In a companion paper, \cite{towsner_worked}, we give an exposition of the motivating ideas in the context of a simpler example.

After introducing the definition of the James space and local unconditionality in Section \ref{sec:definitions}, we outline the usual ultraproduct proof of the local unconditionality of the James space.  In Section \ref{sec:discussion} we return to the issue of metastability, and describe the version of metastability we need for the particular convergence notions which turn up in the proof.  In Section \ref{sec:proof} we give the actual proof of local unconditionality.

%We present our basic definitions in Section \ref{sec:definitions}.  In Section \ref{sec:discussion} we explain the role of proof mining and syntax in the proof, and in particular explain the role of nonstandard analysis in simplifying the proof.  Finally in Section \ref{sec:proof} we complete a constructive, purely standard, proof.  For this we depend on \cite{?}, where we present a constructive, standard proof of a particular theorem about $L_1$-functions which turns out to be the crucial step in the proof.

\section{The James Space}\label{sec:definitions}

\begin{definition}
  The \emph{James space} $J$ \cite{james_1951} consists of infinite sequences $(\alpha_n)$ of real numbers such that:
  \begin{enumerate}
  \item $||(\alpha_n)||_J=\frac{1}{\sqrt{2}}\sup[(\alpha_{p_1}-\alpha_{p_2})^2+(\alpha_{p_2}-\alpha_{p_3})^2+\cdots+(\alpha_{p_{m-1}}-\alpha_{p_m})^2+(\alpha_{p_m}-\alpha_{p_1})^2]^{1/2}$ is bounded, where the supremum ranges over all $m$ and all sequences $p_1<p_2<\cdots<p_m$, and
  \item $\lim_n a_n=0$.
  \end{enumerate}

  The canonical basis for $J$ consists of the vectors $\{e_i\}_{i\in\mathbb{N}}$ where $e_i$ is the sequence $(0,\ldots,0,1,0,\ldots)$ where the $1$ occurs in the $i$-th position.  We write $e^*_i$ for the corresponding dual functionals, so $e^*_j(e_i)=1$ if $i=j$ and $0$ otherwise.
\end{definition}
Note that $||e_i||_J=||e^*_i||_{J^*}=1$.  For each $i$, we define $d_i=\sum_{j\leq i}e_j=(1,\ldots,1,0,\ldots)$.  The sequence $(d_i)$ provides a prototypical example of a sequence in $J$ which is weakly Cauchy but not weakly convergent.

\begin{definition}
If $(c_i)$ is a basis for a Banach space $X$, the \emph{unconditional constant} for $(c_i)$ is the supremum of
\[\frac{||\sum_i \epsilon_i \alpha_i c_i||_X}{||\sum_i \alpha_i c_i||_X}\]
where each $\epsilon_i\in\{-1,1\}$.  The unconditional constant for $X$, $ub(X)$, is the infimum of the unconditional constants of any basis of $X$.
\end{definition}
The definitions above make sense for both finite and infinite dimensional $X$ if we allow for the possibility that the unconditional constant is infinite in the infinite dimensional case.

\begin{definition}
  A Banach space $X$ has \emph{local unconditional structure} if there is a constant $B$ such that every finitely generated subspace of $X$ has a basis with unconditional constant $B$.
\end{definition}

Our main interest is the following theorem:
\begin{theorem}\label{thm:james_noneffective}
  The James space does not have local unconditional structure.
\end{theorem}

The proof (we follow the outline from \cite{MR1474498}) comes from standard facts about the James space (as described in \cite{MR1474498} or \cite{MR0500056}) and two results, one which seems to have first appeared in \cite{MR540367}, though it is due to Johnson and Tzafriri, and the second from \cite{MR0367624}.  (The proof given in the latter uses an argument based on the Hahn-Banach theorem rather than an ultraproduct construction, though the underlying idea is the same.  A proof using ultraproducts explicitly is given, for instance, in \cite{MR1342297}.)

Those two results are quite general, so when specialized to the case of proving the James space does not have local unconditional structure, the proof simplifies:
\begin{proof}[Proof sketch]
  Suppose the James space had a local unconditional constant $B$.  The ultrapower of a space with local unconditional constant is isomorphic to a Banach lattice, so we consider the ultrapower $J^{\mathcal{U}}$ as a Banach lattice.  Consider the Banach lattice closure of the $(d_i)$; call this $X$.  $X$ is separable and isomorphic (as a Banach lattice) to a subspace of $L_1(\Omega)$ for some meausure space $\Omega$.  Let $\pi:X\rightarrow L_1(\Omega)$ be the corresponding injection and let $\pi^*:X^*\rightarrow L_1(\Omega)$ the corresponding dual; note that the range of $\pi^*$ is contained in the dual of range of $\pi$.  In particular, $y^*(x)=\int\pi^*(y^*)\pi(x)d\mu$.

The $L_1$ functions $\pi(d_n)$ converge weakly to some function $f_\infty$ while the functions $\pi^*(e^*_p)$ converge weakly to some function $g_\infty$.  Crucially, the products $\lim_n (\pi(d_n)\pi^*(e^*_p))$ and $\lim_p(\pi(d_n)\pi^*(e^*_p))$ also converge weakly, so the limits exchange:
\begin{align*}
\lim_n \lim_p \int f_n g_p d\mu=&\lim_n \int f_n g_\infty d\mu=\int f_\infty g_\infty d\mu\\
=&\lim_p \int f_\infty g_p d\mu=\lim_p \lim_n \int f_n g_p d\mu.
\end{align*}

But we have
\[\int\pi^*(e^*_p)\pi(d_n)d\mu=\left\{\begin{array}{ll}
1&\text{if }p\leq n\\
0&\text{if }p>n
\end{array}\right.,\]
so $\lim_n \lim_p \int f_n g_p d\mu=0$ while $\lim_p \lim_n \int f_n g_p d\mu=1$, a contradiction.  Therefore our assumption of local unconditionality was false.
\end{proof}

This proof appears to be non-constructive---it tells us that for each constant $B$ there is a sufficiently big finitely generated subspace $X$ of $J$ so that $ub(X)>B$, but it does not tell us what $X$ is, or how big it must be.  Perhaps surprisingly, the techniques in this proof are intrinsically constructive, but conceal the underlying quantitative information (with the benefit of substantially simplifying the proof).  Below we make this quantitative information explicit.

% The main theorem of this paper is the following effective version of Theorem \ref{thm:james_noneffective}:
% \begin{theorem}
%   For any $B$, the subspace of $J$ generated by $\{e_i\}_{i\leq A(B)}$ has unconditional constant greater than $B$.
% \end{theorem}

\section{Proof Mining and Ultraproducts}\label{sec:discussion}

We do not need the literal existence of the functions $f_\infty$ and $g_\infty$ to complete the proof; we really only need the fact that we can exchange the order of the limits in the double limit.  Specifically, we need the following theorem:
\begin{theorem}\label{thm:meta}
  Let $(f_n)_n$ and $(g_p)_p$ be sequences of $L^1$ functions such that
  \begin{itemize}
  \item the sequences $(f_n)_n$ and $(g_p)_p$ converge weakly,
  \item all the functions $f_ng_p$ are $L_1$,
  \item for each fixed $n$, the sequence $(f_ng_p)_p$ converges weakly, and
  \item for each fixed $p$, the sequence $(f_ng_p)_n$ converges weakly.
  \end{itemize}
Then
\[\lim_n\lim_p\int f_ng_p\, d\mu=\lim_p\lim_n\int f_ng_p\, d\mu.\]
\end{theorem}

For our actual application, we do not even need the existence of the limits.  We need only:
\begin{theorem}\label{thm:meta2}
  Let $(f_n)_n$ and $(g_p)_p$ be sequences of $L^1$ functions such that
  \begin{itemize}
  \item the sequences $(f_n)_n$ and $(g_p)_p$ converge weakly,
  \item all the functions $f_ng_p$ are $L_1$,
  \item for each fixed $n$, the sequence $(f_ng_p)_p$ converges weakly, and
  \item for each fixed $p$, the sequence $(f_ng_p)_n$ converges weakly.
  \end{itemize}
Then for every $n$, $p$, and $\epsilon>0$, there exist $m\geq n$ and $q\geq p$ such that for all $l\geq m$ and $r\geq q$, there exist $k\geq l$ and $s\geq r$ such that
\[\left|\int f_mg_s\,d\mu-\int f_lg_q\,d\mu\right|<\epsilon.\]
\end{theorem}
Note that this is a satement with high quantifier complexity: the conclusion has four blocks of alternating quantifiers---in logical notation, it is the following sentence
\[\forall n, p,\epsilon\,\exists m\geq n, q\geq p\,\forall l\geq m,r\geq q\,\exists k\geq l,s\geq r\ \left|\int f_mg_s\,d\mu-\int f_lg_q\,d\mu\right|<\epsilon,\]
which we will call $\sigma$.

In our context, we apply this theorem to an ultraproduct of the finite dimensional subspaces of the James space.  Since this theorem is a true statement about the ultraproduct, there should be some fact about the finite dimensional subspaces of the James space corresponding to it.  It turns out that this is precisely what the metastable version of the statement does for us.  The metastable version of $\sigma$ is:
\begin{quotation}
  For every $\epsilon>0, p, n, \fn{k},\fn{r}$ there are $m\geq n, q\geq p$, $\fn{l},\fn{s}$ such that, if $\fn{k}(m,q,\fn{l},\fn{s})\geq m$ and $\fn{r}(m,q,\fn{l},\fn{s})\geq q$, then:
  \begin{itemize}
  \item $\fn{s}(\fn{k}(m,q,\fn{l},\fn{s}),\fn{r}(m,q,\fn{l},\fn{s}))\geq \fn{r}(m,q,\fn{l},\fn{s})$,
  \item $\fn{l}(\fn{k}(m,q,\fn{l},\fn{s}),\fn{r}(m,q,\fn{l},\fn{s}))\geq \fn{k}(m,q,\fn{l},\fn{s})$,
  \item $\left|\int f_m g_{\fn{s}(\fn{k}(m,q,\fn{l},\fn{s}),\fn{r}(m,q,\fn{l},\fn{s}))}\,d\mu - \int f_{\fn{l}(\fn{k}(m,q,\fn{l},\fn{s}),\fn{r}(m,q,\fn{l},\fn{s}))} g_q\,d\mu\right|<\epsilon$.
  \end{itemize}
\end{quotation}
Notice that this statement has the form
\[\forall \vec x\,\exists \vec y\,\sigma^*(\vec x,\vec y)\]
where $\vec x=\epsilon,p,n,\fn{k},\fn{r}$ and $\vec y=m,q,\fn{l},\fn{s}$.

Then the relationship we have is that $\sigma$ is true in the ultraproduct exactly when
\begin{quotation}
  For every $\vec x$ there is a $\vec Y$ so that whenever $J_K$ is a $K$-dimensional subspace of $J$ with $K$ sufficiently large, there is a $\vec y\leq \vec Y$ such that $\sigma^*(\vec x,\vec y)$ is true in $J_K$.
\end{quotation}
In other words, the truth of $\sigma$ in the ultraproduct is equivalent to the ``uniform truth'' of $\sigma^*$ in the original structures.  (We avoid the technical issue of what it means to have $\vec y\leq \vec Y$ in general, given that $\vec Y$ involves functions \cite{MR1428007,MR2156133,MR2853725}, since we will only need a special case.)

$\sigma^*$ is already rather complicated; fortunately, we only need the case where $n=p=0$, $\fn{k}(m,q,\fn{l},\fn{s})=\max\{m,q+1\}$ and $\fn{r}(m,q,\fn{l},\fn{s})=\max\{m+1,q\}$.  In this case the conclusion becomes
\begin{quotation}
  For every $\epsilon>0$ there are $m, q$, $\fn{l},\fn{s}$ such that:
  \begin{itemize}
  \item $\fn{s}(\max\{m,q+1\},\max\{m+1,q\})\geq \max\{m,q+1\}$,
  \item $\fn{l}(\max\{m,q+1\},\max\{m+1,q\})\geq \max\{m+1,q\}$,
  \item $\left|\int f_mg_{\fn{s}(\max\{m,q+1\},\max\{m+1,q\})}\,d\mu - \int f_{\fn{l}(\max\{m,q+1\},\max\{m+1,q\})}g_q\,d\mu\right|<\epsilon$.
  \end{itemize}
\end{quotation}
Bounds on the sizes of the values $\fn{s}(\max\{m,q+1\},\max\{m+1,q\})$ and $\fn{l}(\max\{m,q+1\},\max\{m+1,q\})$ depend on the assumptions, however.  For instance, we expect the size of $\fn{s}$ to depend on how rapidly the sequences $(f_n)_n$ and $(g_p)_p$ converge.  More precisely, we expect the size of $\fn{s}$ to depend on the rate of metastable convergence.

It will turn out that the sequences we need converge in a very strong way: they have bounded fluctuations.\footnote{Saying that a sequence converges metastably is be equivalent to saying that a certain tree is well-founded.  Having bounded fluctuations is the strengthening in which the tree specifically has height at most $\omega$.  \cite{gaspar} considers a similar issue.}
\begin{definition}
  The sequence $(a_n)$ has \emph{bounded fluctuations with bound $f(\epsilon)$} if for every $\epsilon,\fn{m},n$ there is an $m\in[n,\fn{m}^{f(\epsilon)}(n)]$ such that whenever $k,k'\in[m,\fn{m}(k)]$, $|a_k-a_{k'}|<\epsilon$.
\end{definition}

This makes it possible to apply the main quantitative result from \cite{towsner:abs_cont}.  The result there is stated in terms of the \emph{fast-growing hierarchy} of functions:
\begin{itemize}
\item $f_0(n)=n+1$,
\item $f_{m+1}(n)=f^n_m(n)$.
\end{itemize}
The exponent means that $f_m$ is applied $n$ consecutive times to $n$, so $f_1(n)=2n$, $f_2(n)=2^nn$, and so on.  These functions grow very rapidly; in particular the function $f_\omega(m)=f_m(m)$ grows at roughly the same speed as the Ackermann function.  This lets us state:
\begin{theorem}\label{thm:cited}
Suppose that:
  \begin{itemize}
\item Each sequence $(f_ng_p)_p$ (for fixed $n$) and $(f_ng_p)_n$ (for fixed $p$) has bounded fluctuations with the uniform bound $8 B^2(\lceil 1/\epsilon\rceil)^2$,
\item   For each $n$ and any $\sigma\subseteq\Omega$ with $\mu(\sigma)<\epsilon/B 2^n$, $\int_\sigma|f_n|d\mu<\epsilon$,
\item   For each $p$ and any $\sigma\subseteq\Omega$ with $\mu(\sigma)<\epsilon/B 2^p$, $\int_\sigma|g_p|d\mu<\epsilon$,
\item For each $n$, $||f_n||_{L^1}\leq B$,
\item For each $p$, $||g_p||_{L^1}\leq B$,
\end{itemize}
Then for every $E$ there exist $m<s$ and $q<l$ such that:
  \begin{itemize}
  \item $s, l \leq f_{\omega}(2^{22}B^4\lceil 1/\epsilon\rceil^4+5)$, and
  \item $|(f_{m}g_{s})(\Omega)-(f_{l}g_{q})(\Omega)|<20\epsilon$.
  \end{itemize}
\end{theorem}
The theorem as stated in \cite{towsner:abs_cont} has an additional technical condition regarding partitions into approximate level sets which is trivial in this case because $f_n,g_p$ are explicitly presented as functions.

\section{The James Space as a Finite Measure Space}\label{sec:proof}
\subsection{Bounds on Fluctuations}
Recall the sequence $d_i=\sum_{j\leq i}e_i=(1,\ldots,1,0,\ldots)$ in the James space where the first $i$ elements of the sequence are $1$ and the rest are $0$.  Since the sequence $(d_i)$ is weakly Cauchy, we should be able to obtain bounds on the metastable weak convergence of this sequence.  In this case we obtain an even stronger bound, which we will later convert this into a proof that the functions we are interested in have bounded fluctuations.
\begin{lemma}\label{thm:james_metaconvergence}
For any $\epsilon>0$, any $k\geq 2\lceil 1/\epsilon\rceil^2$, any $n_0<\cdots<n_{k}$ and any $y^*$ with $||y^*||_{J^*}\leq 1$, there is an $i<k$ so that $|y^*(d_{n_i})-y^*(d_{n_{i+1}})|< \epsilon$.
\end{lemma}
\begin{proof}
Towards a contradiction, suppose not, and let $\epsilon,y^*,n_0<\cdots<n_{k}$ be a counterexample.  We may divide the index set $I=[0,k)$ into two components, $I^>=\{i<k\mid y^*(d_{n_i})>y^*(d_{n_{i+1}})\}$ and $I^<=[0,k)\setminus I^>=\{i<k\mid y^*(d_{n_i})<y^*(d_{n_{i+1}})\}$.  Clearly we have either $|I^>|> k/2$ or $|I^<|> k/2$; without loss of generality, we assume $|I^<|> k/2$ (the other case is symmetric).

Set $\hat x=\sum_{i\in I^>}(d_{n_{i+1}}-d_{n_i})$.  So $\hat x$ is the sum of those $e_j$ such that $j\in\bigcup_{i\in I^>}(n_i,n_{i+1}]$.  Therefore $||\hat x||_J=\frac{1}{\sqrt{2}}\sqrt{2|I^>|}=\sqrt{|I^>|}$ while
\[y^*(\hat x)=\sum_{i\in I^>}y^*(d_{n_{i+1}})-y^*(d_{n_i})\geq |I^>|\epsilon.\]
But this means $y^*(\hat x)\geq\sqrt{|I^>|}\epsilon||\hat x||_J> \sqrt{\lceil 1/\epsilon\rceil^2}\epsilon||\hat x||_J\geq ||\hat x||_J$, contradicting the fact that $||y^*||\leq 1$.
\end{proof}

Analagously, consider the functional $e^*_\infty=\lim_{i\rightarrow\infty}e^*_i$; on $J$ this is of course the functional which is constantly $0$, but it becomes more useful on the ultrapower of $J$.  The fact that $e^*_\infty$ is actually well-defined on the ultrapower of $J$ is equivalent to the fact that the sequence $\lim_{i\rightarrow\infty}e^*_i$ converges metastably as in the following lemma:
\begin{lemma}\label{thm:james_dual_metaconvergence}
For any $\epsilon>0$, any $k\geq 2\lceil 1/\epsilon\rceil^2$, any $p_0<\cdots<p_{k}$ and any $x$ with $||x||_J\leq 1$, there is an $i<k$ so that $|e^*_{p_i}(x)-e^*_{p_{i+1}}(x)|< \epsilon$.
\end{lemma}
\begin{proof}
  Follows immediately from the definition of the norm $||\cdot||_J$: $x$ is some sequence $(x_n)$.  If $|e^*_{p_i}(x)-e^*_{p_{i+1}}(x)|>\epsilon$ for all $i<k$ then by definition
\begin{align*}
||x||_J
&\geq\frac{1}{\sqrt{2}}\sqrt{\sum_{i<k}(x_{p_i}-x_{p_{i+1}})^2}\\
&=\frac{1}{\sqrt{2}}\sqrt{\sum_{i<k}(e^*_{p_i}(x)-e^*_{p_{i+1}}(x))^2}\\
&>\frac{1}{\sqrt{2}}\sqrt{k\lceil 1/\epsilon\rceil^2}\\
&\geq1.
\end{align*}
\end{proof}

\subsection{A Finite Measure Space}

Suppose the subspace $J_K$ generated by $\{e_i\}_{i\leq K}$ has a basis $(\omega_i)_{i\leq K}$ with unconditional constant $B$.
(Our use of the letter $\omega$ presages the fact that we will mostly be concerned with viewing the $\omega_i$ as elements in a measure space.)  Let $\gamma^*_i$ be the dual functionals corresponding to this basis, so any $x\in J_K$ satisfies $x=\sum_i\gamma^*_i(x)\omega_i$.

We can view $(\omega_i)_{i\leq K}$ as inducing a Banach lattice structure, with $x\leq y$ if for each $i\leq K$, $\gamma^*_i(x)\leq \gamma^*_i(y)$.  We will not need this structure itself, but it motivates the following definitions.

For $x\in J_K$ we define $|x|=\sum_i|\gamma^*_i(x)|\omega_i$; since the $(\omega_i)$ are an unconditional basis, $||x||_J/B\leq ||\,|x|\,||_J\leq B||x||_J$.  Similarly, for $x^*\in J^*_K$ we define $|x^*|(x)=\sum_i\gamma^*_i(x)|x^*(\omega_i)|$; we have $||x^*||_{J^*}/B\leq||\,|x^*|\,||_{J^*}\leq B||x^*||_{J^*}$.  Note that if we choose $\epsilon_i\in\{-1,1\}$ so that for each $i$, $\epsilon_i\gamma^*_i(x)x^*(\omega_i)$ is non-negative then we may set $x'=\sum_i\epsilon_i\gamma^*_i(x)\omega_i$ and we have
\[0\leq |x^*|(|x|)=x^*(x')\leq ||x^*||_{J^*}||x'||_J\leq B||x^*||_{J^*}||x||_J.\]

We fix canonical elements $d=\sum_{j\leq K} 2^{-j-1}|d_j|$ and $d^*=\sum_{j\leq K}2^{-j-1} |e^*_j|$.  Observe that 
\[d^*(d)=\sum_{j,j'\leq K}2^{-j-j'-2}|e^*_j|(|d_{j'}|)\leq\sum_{j,j'\leq K}2^{-j-j'-2} B=B.\]
On the other hand,
\[d^*(d)\geq\sum_j 2^{-2j-2}|e^*_j(d_j)|\geq\frac{1}{4}.\]

We define a finite measure space, $(\Omega,\mu)$; we take $\Omega=\{\omega_i\}_{i\leq K}$, and since $\Omega$ is atomic, it suffices to define $\mu(\{\omega_i\})=\frac{\gamma^*_i(d)}{d^*(d)}d^*(\omega_i)$.  We have
\[\mu(\Omega)=\frac{1}{d^*(d)}\sum_i\gamma^*_i(d)d^*(\omega_i)=\frac{d^*(d)}{d^*(d)}=1.\]

We now define an embedding $\pi:J_K\rightarrow L_1(\Omega)$ by setting $\pi(x)$ to be the function
\[\sum_i\frac{d^*(d)}{\gamma^*_i(d)}\gamma^*_i(x)\chi_i\]
where $\chi_i$ is the characteristic function of the set $\{\omega_i\}$.  That is, $\pi(x)$ is the function which, at the point $\omega_i$, takes the value $\frac{d^*(d)}{\gamma^*_i(d)}\gamma^*_i(x)$.  This definition has the convenient property that
\[||\pi(x)||_1=\sum_i\frac{|\gamma^*_i(x)|d^*(d)}{\gamma^*_i(d)}\frac{\gamma^*_i(d)d^*(\omega_i)}{d^*(d)}=\sum_i|\gamma^*_i(x)|d^*(\omega_i)=d^*(|x|)\leq B||x||_J.\]
We also have an embedding $\pi^*:J_K^*\rightarrow L_1(\Omega)$ given by setting $\pi^*(x^*)$ to be the function
\[\sum_i\frac{d^*(d)}{d^*(\omega_i)}x^*(\omega_i) \chi_i.\]
Therefore
\begin{align*}
\int\pi^*(x^*)\pi(x)d\mu
&=\sum_i \frac{x^*(\omega_i)d^*(d)}{d^*(\omega_i)}\frac{\gamma^*_i(x) d^*(d)}{\gamma^*_i(d)}\frac{\gamma^*_i(d)d^*(\omega_i)}{d^*(d)}\\
&=\sum_i\gamma^*_i(x)x^*(\omega_i)d^*(d)\\
&=x^*(x)d^*(d)\\
&\in[x^*(x)/4,Bx^*(x)]
\end{align*}
and
\[||\pi^*(x^*)||_1=\sum_i\frac{|x^*(\omega_i)|d^*(d)}{d^*(\omega_i)}\frac{\gamma_i^*(d)d^*(\omega_i)}{d^*(d)}
=\sum_i|x^*(\omega_i)|\gamma^*_i(d)=|x^*|(d)\leq B||x^*||_{J^*}.\]

We define two sequences of $L^1(\Omega)$ functions: we set $f_n=\pi(d_n)$ and $g_p=\pi^*(e^*_p)$.  Note that for any $n$ we have $||f_n||_1\leq B$ and for any $p$, $||g_p||_1\leq B$.

We also have
\[||f_n||_\infty=\sup_i \left|\frac{\gamma^*_i(d_n)d^*(d)}{\gamma^*_i(d)}\right|\leq\sup_i\left|\frac{\gamma^*_i(d_n)d^*(d)}{2^{-n}\gamma^*_i(|d_n|)}\right|\leq B2^n\]
and
\[||g_p||_\infty=\sup_i\left|\frac{e^*_p(\omega_i)d^*(d)}{d^*(\omega_i)}\right|
\leq\sup_i\left|\frac{e^*_p(\omega_i)d^*(d)}{2^{-p}|e^*_p(\omega_i)|}\right|
\leq B2^p.\]

\subsection{Quantitative Convergence and Continuity}

\begin{lemma}
  For any fixed $p$, the sequence $(\rho_n\lambda_p)_n$ has bounded fluctuations with bound $8B^2\lceil 1/\epsilon\rceil^2$.
\end{lemma}
\begin{proof}
Let $\epsilon>0,\fn{m},n,\sigma$ be given.  Without loss of generality, assume $\fn{m}(n)\leq\fn{m}(n+1)$ for all $n$ (if this is not the case, we replace $\fn{m}$ with $\fn{m}'(n)=\max_{n'\leq n}\fn{m}(n')$).  

Consider the function $y^*(x)=\frac{1}{B}\int_\sigma\pi(x)\lambda_p\,d\mu$ and define a sequence inductively by $m_0=n$ and if there is any $m\in[m_i,\fn{m}(m_i)]$ such that $|y^*(d_{m_i})-y^*(d_m)|\geq \epsilon/2B$ then $m_{i+1}$ is the least such $m$, otherwise $m_{i+1}=\fn{m}(m_i)$.  By the monotonicity of $\fn{m}$, $m_{2B^2E^2}\leq\fn{m}^{2B^2E^2}(n)$.  By Lemma \ref{thm:james_metaconvergence} there is an $i<8B^2\lceil 1/\epsilon\rceil^2$ so that $|y^*(d_{m_i})-y^*(d_{m_{i+1}})|< \epsilon/2B$.  By the choice of $m_{i+1}$, it must be that $m_{i+1}=\fn{m}(m_i)$ and for every $k\in[m_i,\fn{m}(m_i)]$, $|y^*(d_{m_i})-y^*(d_k)|\geq\epsilon/2B$.

Suppose there were some $k,k'\in[m_i,\fn{m}(m_i)]$ so that $|y^*(d_k)-y^*(d_{k'})|\geq \epsilon/B$.  Then either $|y^*(d_{m_i})-y^*(d_k)|\geq\epsilon/2B$ or $|y^*(d_{m_i})-y^*(d_{k'})|\geq\epsilon/2B$.  But this is a contradiction.  Therefore for every $k,k'\in[m_i,\fn{m}(m_i)]$,
\[|\rho_{m_i}(\sigma)-\rho_m(\sigma)|=B|y^*(d_{k})-y^*(d_{k'})|<\epsilon.\]
\end{proof}

Similarly, using Lemma \ref{thm:james_dual_metaconvergence}, we obtain
\begin{lemma}
  For any fixed $n$, the sequence $(\rho_n\lambda_p)_p$ has bounded fluctuations with bound $8B^2\lceil 1/\epsilon\rceil^2$.
\end{lemma}

\begin{lemma}
  For each $n$ and any $\sigma\subseteq\Omega$ with $\mu(\sigma)<\epsilon/B2^n$, $\int_\sigma|\rho_n|d\mu<\epsilon$.
\end{lemma}
\begin{proof}
  Immediate since $||\rho_n||_\infty\leq B2^n$.
\end{proof}

\begin{lemma}
  For each $p$ and any $\sigma\subseteq\Omega$ with $\mu(\sigma)<\epsilon/B2^p$, $\int_\sigma|\lambda_p|d\mu<\epsilon$.
\end{lemma}

\subsection{Putting it Together}

\begin{theorem}
  If $K\geq f_\omega(2^{29}B^4+5)$ then the local unconditionality constant of a basis of $J_K$ is $>B$.
\end{theorem}
\begin{proof}
  Suppose $J_K$ has a basis $(\omega_i)_{i\leq K}$ with unconditional constant $\leq B$.  We construct $f_n, g_p$ as described above, and apply Theorem \ref{thm:cited} with $\epsilon=1/80$ to obtain $m<s\leq K$ and $q<l\leq K$ so that
\[|\int(f_mg_s)d\mu-\int(f_lg_q)d\mu|<1/4.\]
But since $m<s$, $\int(f_mg_s)d\mu=\int\pi(d_m)\pi^*(e^*_s)d\mu=e^*_s(d_m)d^*(d)=0$.  On the other hand, since $q<l$, $\int(f_lg_q)d\mu=e^*_q(d_l)d^*(d)\geq 1/4$, which is a contradiction.
\end{proof}

\bibliographystyle{plain}
\bibliography{../../Bibliographies/main}

\begin{thebibliography}{10}

\bibitem{avigad:MR1640329}
Jeremy Avigad and Solomon Feferman.
\newblock G\"odel's functional (``{D}ialectica'') interpretation.
\newblock In {\em Handbook of proof theory}, volume 137 of {\em Stud. Logic
  Found. Math.}, pages 337--405. North-Holland, Amsterdam, 1998.

\bibitem{avigad:MR2550151}
Jeremy Avigad, Philipp Gerhardy, and Henry Towsner.
\newblock Local stability of ergodic averages.
\newblock {\em Trans. Amer. Math. Soc.}, 362(1):261--288, 2010.

\bibitem{MR1342297}
Joe Diestel, Hans Jarchow, and Andrew Tonge.
\newblock {\em Absolutely summing operators}, volume~43 of {\em Cambridge
  Studies in Advanced Mathematics}.
\newblock Cambridge University Press, Cambridge, 1995.

\bibitem{MR2156133}
Fernando Ferreira and Paulo Oliva.
\newblock Bounded functional interpretation.
\newblock {\em Ann. Pure Appl. Logic}, 135(1-3):73--112, 2005.

\bibitem{MR1474498}
Helga Fetter and Berta Gamboa~de Buen.
\newblock {\em The {J}ames forest}, volume 236 of {\em London Mathematical
  Society Lecture Note Series}.
\newblock Cambridge University Press, Cambridge, 1997.
\newblock With a foreword by Robert C. James and a prologue by Bernard
  Beauzamy.

\bibitem{MR0367624}
T.~Figiel, W.~B. Johnson, and L.~Tzafriri.
\newblock On {B}anach lattices and spaces having local unconditional structure,
  with applications to {L}orentz function spaces.
\newblock {\em J. Approximation Theory}, 13:395--412, 1975.
\newblock Collection of articles dedicated to G. G. Lorentz on the occasion of
  his sixty-fifth birthday, IV.

\bibitem{gaspar}
Jaime Gaspar and Ulrich Kohlenbach.
\newblock On {T}ao's ``finitary'' infinite pigeonhole principle.
\newblock {\em J. Symbolic Logic}, 75(1):355--371, 2010.

\bibitem{16329}
Bill~Johnson (http://mathoverflow.net/users/2554/bill johnson).
\newblock How helpful is non-standard analysis?
\newblock MathOverflow.
\newblock URL:http://mathoverflow.net/q/16329 (version: 2010-02-25).

\bibitem{james_1951}
Robert~C. James.
\newblock A non-reflexive banach space isometric with its second conjugate
  space.
\newblock {\em Proceedings of the National Academy of Sciences of the United
  States of America}, 37(3):pp. 174--177, 1951.

\bibitem{kohlenbach:MR2445721}
U.~Kohlenbach.
\newblock {\em Applied proof theory: proof interpretations and their use in
  mathematics}.
\newblock Springer Monographs in Mathematics. Springer-Verlag, Berlin, 2008.

\bibitem{MR1195271}
Ulrich Kohlenbach.
\newblock Effective bounds from ineffective proofs in analysis: an application
  of functional interpretation and majorization.
\newblock {\em J. Symbolic Logic}, 57(4):1239--1273, 1992.

\bibitem{MR1428007}
Ulrich Kohlenbach.
\newblock Analysing proofs in analysis.
\newblock In {\em Logic: from foundations to applications ({S}taffordshire,
  1993)}, Oxford Sci. Publ., pages 225--260. Oxford Univ. Press, New York,
  1996.

\bibitem{MR2853725}
Ulrich Kohlenbach.
\newblock A note on the monotone functional interpretation.
\newblock {\em MLQ Math. Log. Q.}, 57(6):611--614, 2011.

\bibitem{MR0500056}
Joram Lindenstrauss and Lior Tzafriri.
\newblock {\em Classical {B}anach spaces. {I}}.
\newblock Springer-Verlag, Berlin, 1977.
\newblock Sequence spaces, Ergebnisse der Mathematik und ihrer Grenzgebiete,
  Vol. 92.

\bibitem{MR540367}
Joram Lindenstrauss and Lior Tzafriri.
\newblock {\em Classical {B}anach spaces. {II}}, volume~97 of {\em Ergebnisse
  der Mathematik und ihrer Grenzgebiete [Results in Mathematics and Related
  Areas]}.
\newblock Springer-Verlag, Berlin, 1979.
\newblock Function spaces.

\bibitem{tao06}
Terence Tao.
\newblock A quantitative ergodic theory proof of {S}zemer\'edi's theorem.
\newblock {\em Electron. J. Combin.}, 13(1):Research Paper 99, 49 pp.
  (electronic), 2006.

\bibitem{tao:MR2408398}
Terence Tao.
\newblock Norm convergence of multiple ergodic averages for commuting
  transformations.
\newblock {\em Ergodic Theory Dynam. Systems}, 28(2):657--688, 2008.

\bibitem{towsner:abs_cont}
Henry Towsner.
\newblock Towards an effective theory of absolutely continuous measures.
\newblock draft.

\bibitem{towsner_worked}
Henry Towsner.
\newblock A worked example of the functional interpretation.
\newblock draft.

\bibitem{MR2964881}
Benno van~den Berg, Eyvind Briseid, and Pavol Safarik.
\newblock A functional interpretation for nonstandard arithmetic.
\newblock {\em Ann. Pure Appl. Logic}, 163(12):1962--1994, 2012.

\end{thebibliography}
\end{document}